\documentclass{article}
\usepackage{fullpage}
\usepackage[usenames]{color}
\usepackage{bm,bbm,amsmath,amssymb,amsthm,graphicx}
\usepackage[caption=false,font=normalsize,labelfont=sf,textfont=sf]{subfig}
\usepackage[title]{appendix}
\usepackage{algorithm,pifont}
\usepackage{algpseudocode}
\usepackage{todonotes}
\usepackage{hyperref}
\usepackage{stmaryrd}
\usetikzlibrary{calc,patterns,decorations.pathmorphing,decorations.markings}

\allowdisplaybreaks

\providecommand{\keywords}[1]{\textbf{\textit{Index terms---}} #1}

\numberwithin{equation}{section}

\newcommand{\MAT}{\left[ \begin{array}}  
\newcommand{\mat}{\end{array} \right]}
\newtheorem{Definition}{Definition}[section]
\newtheorem{Corollary}[Definition]{Corollary}
\newtheorem{Lemma}[Definition]{Lemma}
\newtheorem{Theorem}[Definition]{Theorem}

\newtheorem{Q}[Definition]{Question}

\newtheorem{remark}[Definition]{Remark}

\def \R {\mathbbm{R}}

\def \w {\bm{w}}

\def \x {{\bf x}}

\def \y {{\bf y}}

\begin{document}

\title{On Outer Bi-Lipschitz Extensions of Linear Johnson-Lindenstrauss Embeddings of Subsets of $\mathbbm{R}^N$}

\author{Rafael Chiclana\thanks{Department of Mathematics, Michigan State University, East Lansing, MI 48824.  {\bf E-mails:}  chiclan1@msu.edu, roachma3@msu.edu}, Mark A. Iwen\thanks{Department of Mathematics, and Department of Computational Mathematics, Science and Engineering (CMSE), Michigan State University, East Lansing, MI 48824.  {\bf E-mail:} iwenmark@msu.edu}, and Mark Philip Roach\footnotemark[1]}
%

\maketitle
\begin{abstract}
The celebrated Johnson-Lindenstrauss lemma states that for all $\varepsilon \in (0,1)$ and finite sets $X \subseteq \mathbbm{R}^N$ with $n>1$ elements, there exists a matrix $\Phi \in \mathbbm{R}^{m \times N}$ with $m=\mathcal{O}(\varepsilon^{-2}\log n)$ such that
\[ (1 - \varepsilon) \|\x-\y\|_2 \leq \|\Phi\x-\Phi\y\|_2 \leq (1+\varepsilon)\|\x-\y\|_2 \quad \forall\, \x,\y \in X.\]
Herein we consider so-called ``terminal embedding" results which have recently been introduced in the computer science literature as stronger extensions of the Johnson–Lindenstrauss lemma for finite sets.  After a short survey of this relatively recent line of work, we extend the theory of terminal embeddings to hold for arbitrary (e.g., infinite) subsets $X \subseteq \mathbbm{R}^N$, and then specialize our generalized results to the case where $X$ is a low-dimensional compact submanifold of $\mathbbm{R}^N$.  In particular, we prove the following generalization of the Johnson–Lindenstrauss lemma:  For all $\varepsilon \in (0,1)$ and $X\subseteq\mathbbm{R}^N$, there exists a terminal embedding $f: \mathbbm{R}^N \longrightarrow \mathbbm{R}^{m}$ such that
$$(1 - \varepsilon) \| {\bf x} - {\bf y} \|_2 \leq \left\| f({\bf x}) - f({\bf y}) \right\|_2 \leq (1 + \varepsilon) \| {\bf x} - {\bf y} \|_2 \quad \forall \, {\bf x} \in X ~{\rm and}~ \forall \, {\bf y} \in \mathbbm{R}^N.$$
Crucially, we show that the dimension $m$ of the range of $f$ above is optimal up to multiplicative constants, satisfying $m=\mathcal{O}(\varepsilon^{-2} \omega^2(S_X))$, where $\omega(S_X)$ is the Gaussian width of the set of unit secants of $X$, $S_X=\overline{\{(\x-\y)/\|\x-\y\|_2 \colon \x\neq \y \in X\}}$.  Furthermore, our proofs are constructive and yield algorithms for computing a general class of terminal embeddings $f$, an instance of which is demonstrated herein to allow for more accurate compressive nearest neighbor classification than standard linear Johnson-Lindenstrauss embeddings do in practice.
\end{abstract}


\section{Introduction}

The classical Kirszbraun theorem \cite{kirszbraun1934uber} ensures that a Lipschitz continuous function $f: X \rightarrow \mathbbm{R}^m$, where $X$ is a subset of $\mathbbm{R}^N$, can always be extended to a function $f': \mathbbm{R}^N \rightarrow \mathbbm{R}^m$ with the same Lipschitz constant as $f$. More recently, similar results have been proven for bi-Lipschitz functions in the theoretical computer science literature. Recall that the \textbf{distortion} of a bi-Lipschitz map $f \colon X \longrightarrow \mathbbm{R}^m$ is the minimum $D\geq 0$ such that for some $\lambda >0$
\[ \lambda \|\x-\y\|_2 \leq \|f(\x)-f(\y)\|_2\leq \lambda D \|\x-\y\|_2 \quad \forall \, \x , \y \in X,\]
where $\|\cdot\|_2$ represents the Euclidean norm. Let $X\subseteq Y \subseteq \mathbbm{R}^N$, $m\leq m' \in \mathbb{N}$, and consider $f\colon  X\longrightarrow \mathbbm{R}^m$. We say that $f' \colon Y \longrightarrow \mathbbm{R}^{m'}$ is an \textbf{outer extension} of $f$ if $f(\x)=f'(\x)$ for all $\x \in X$, where we identify points $(\x_1,\ldots,\x_m) \in \mathbbm{R}^m$ and $(\x_1,\ldots,\x_m,0,\ldots,0) \in \mathbbm{R}^{m'}$. Mahabadi et al. proved in \cite{mahabadi2018nonlinear} that every bi-Lipschitz map $f \colon X \longrightarrow \mathbbm{R}^m$ with distortion $D$ has an outer extension $f'\colon \mathbbm{R}^N \longrightarrow \mathbbm{R}^{N+m}$ with distortion at most $3D$. Moreover, they showed that when the map $f$ is a Johnson-Lindenstrauss embedding, the dimension $N+m$ can be improved to $m+1$ in exchange of obtaining distortion $D^2$. Recall that the Johnson-Lindenstrauss lemma \cite{johnson1984extensions} states that for $\varepsilon \in (0,1)$ and a finite set $X \subseteq \mathbbm{R}^N$ with $n>1$ elements, there exists a matrix $\Phi \in \mathbbm{R}^{m \times N}$ with $m=\mathcal{O}(\varepsilon^{-2}\log n)$ such that
\[ (1-\varepsilon)\|\x-\y\|_2 \leq \|\Phi\x-\Phi\y\|_2 \leq (1+\varepsilon)\|\x-\y\|_2 \quad \forall\, \x,\y \in X.\]
Moreover, it has been recently shown that the dimension $m$ of the Euclidean space where $X$ is embedded is optimal (see \cite{larsen2017optimality}). In this work, we focus on the study of the following stronger notion of embedding. We say that $f \colon \mathbbm{R}^N \longrightarrow \mathbbm{R}^m$ is a \textbf{terminal embedding} for $X \subseteq \mathbbm{R}^N$ with distortion $\varepsilon>0$ if
\begin{equation}\label{eq: terminal embedding} 
	(1-\varepsilon)\|\x-\y\|_2 \leq \|f(\x)-f(\y)\|_2 \leq (1+\varepsilon)\|\x-\y\|_2 \quad \forall \, \x \in X \quad \forall \, \y \in \mathbbm{R}^N.
\end{equation}
This class of embeddings was introduced by Elkin, Filtser, and Neiman in \cite{elkin2017terminal}, where terminal embeddings for a finite $X\subseteq \mathbbm{R}^N$ of bounded distortion are constructed. This was improved by Mahabadi et al. in \cite{mahabadi2018nonlinear}, where the authors consider outer extensions of Johnson-Lindenstrauss embeddings to construct terminal embeddings with arbitrarily small distortion. The target dimension obtained in \cite{mahabadi2018nonlinear} was $m=\mathcal{O}(\varepsilon^{-4}\log n)$, where $n$ is the cardinality of the finite set $X$. Later on, this was improved by Narayan and Nelson in \cite{narayanan2019optimal}, where the optimal dimension $m=\mathcal{O}(\varepsilon^{-2}\log n)$ was obtained. These contributions ultimately led to the following remarkable result: Let $\varepsilon \in (0,1)$ and let $X \subseteq \mathbbm{R}^N$ be a finite set of size $n>1$. There exists $f \colon \mathbbm{R}^N \longrightarrow \mathbbm{R}^m$ with $m=\mathcal{O}(\varepsilon^{-2} \log n)$ such that
\[ (1-\varepsilon)\|\x-\y\|_2 \leq \|f(\x)-f(\y)\|_2 \leq (1+\varepsilon)\|\x-\y\|_2 \quad \forall \, \x \in X \quad \forall \, \y \in \mathbbm{R}^N.\]

In this paper we present a generalization of this result that holds for arbitrary subsets $X \subseteq \mathbbm{R}^N$ where the role that the cardinality of $X$ plays in the target dimension $m$ is instead replaced by the Gaussian width of the set of the unit secants of $X$.  
In the special case where $X=\mathcal{M}$ is a submanifold of $\mathbbm{R}^N$, recent results bounding the requisite Gaussian width in terms of other fundamental geometric quantities (e.g., their reach, dimension, volume, etc.) \cite{iwen2024on} can then be brought to bear in order to produce terminal embeddings of $\mathcal{M}$ into $\mathbbm{R}^m$ satisfying (\ref{eq: terminal embedding}) with a target dimension $m$ of optimal order.\footnote{Lower bounds on achievable embedding dimensions $m$ for submanifolds of $\mathbbm{R}^N$ are discussed in \cite{iwen2023lower}.}  Moreover, our proofs are constructive and provide us with an algorithm that can be used to generate terminal embeddings in practice.  Motivating applications of terminal embeddings of submanifolds of $\mathbbm{R}^N$ related to compressive classification via manifold models \cite{davenport2007the} are discussed next.

\subsection{Universally Accurate Compressive Classification via Noisy Manifold Data}
\label{sec:motivating_application}

It is one of the sad facts of life that most everyone eventually comes to accept:  everything living must eventually die, you can't always win, you aren't always right, and -- worst of all to the most dedicated of data scientists -- there is always noise contaminating your datasets.  Nevertheless, there are mitigating circumstances and achievable victories implicit in every statement above -- most pertinently here, there are mountains of empirical evidence that noisy training data still permits accurate learning.  In particular, when the noise level is not too large, the mere existence of a low-dimensional data model which only approximately fits your noisy training data can still allow for successful, e.g., nearest-neighbor classification using only a highly compressed version of your original training dataset (even when you know very little about the model specifics) \cite{davenport2007the}.  Better quantifying these empirical observations in the context of low-dimensional manifold models is the primary motivation for our main manifold result below.

For example, let $\mathcal{M} \subseteq \mathbbm{R}^N$ be a $d$-dimensional submanifold of $\mathbbm{R}^N$ (our data model), fix $\delta \in \mathbbm{R}^+$ (our effective noise level), and choose $T \subseteq {\rm tube}(\delta,\mathcal{M}) := \left \{ {\bf x} ~|~ \exists {\bf y} \in \mathcal{M} ~{\rm with}~ \| {\bf x} - {\bf y} \|_2 \leq \delta \right\}$ (our ``noisy'' and potentially high-dimensional training data).  Fix $\varepsilon \in (0,1)$.  For a terminal embedding $f: \mathbbm{R}^N \rightarrow \mathbbm{R}^m$ of $\mathcal{M}$ as per (\ref{eq: terminal embedding}), one can use the triangle inequality to see that
\begin{equation}
 (1-\varepsilon)\left\| {\bf z} - {\bf t} \right\|_2 -2(1+\varepsilon) \delta \leq \left\| f({\bf z}) - f( {\bf t}) \right\|_2 \leq (1 + \varepsilon) \left\| {\bf z} - {\bf t} \right\|_2 + 2(1+\varepsilon) \delta 
\label{equ:DisttoMan}
\end{equation}
will hold simultaneously for all ${\bf t} \in T$ and ${\bf z} \in \mathbbm{R}^N$, {\it where $f$ has an embedding dimension that only depends on the geometric properties of ${\mathcal M}$ (and not necessarily on $|T|$)}. Thus, if $T$ includes a sufficiently dense external cover of $\mathcal{M}$, then $f$ will allow us to approximate the distance of all ${\bf z} \in \mathbbm{R}^N$ to $\mathcal{M}$ in the compressed embedding space via the estimator 
\begin{equation}
\tilde{d}(f({\bf z}),f(T)) := \inf_{{\bf t } \in T} \left\| f({\bf z}) - f( {\bf t}) \right\|_2 \approx d({\bf z},\mathcal{M}) := \inf_{{\bf y} \in \mathcal{M}} \| {\bf z} - {\bf y} \|_2
\label{equ:CompressedEstimator}
\end{equation} 
up to $\mathcal{O}(\delta)$-error.  As a result, if one has noisy data from two disjoint manifolds $\mathcal{M}_1,\mathcal{M}_2 \subseteq \mathbbm{R}^N$, one can use this compressed $\tilde{d}$ estimator to correctly classify all data ${\bf z} \in {\rm tube}(\delta,\mathcal{M}_1) \bigcup {\rm tube}(\delta,\mathcal{M}_2)$ as being in either $T_1 := {\rm tube}(\delta,\mathcal{M}_1)$ (class 1) or $T_2 := {\rm tube}(\delta,\mathcal{M}_2)$ (class 2) as long as $\displaystyle \inf_{{\bf x} \in T_1, {\bf y} \in T_2} \| {\bf x} - {\bf y} \|_2$ is sufficiently large.  In short, terminal manifold embeddings demonstrate that accurate compressive nearest-neighbor classification based on noisy manifold training data is always possible as long as the manifolds in question are sufficiently far apart (though not necessarily separable from one another by, e.g., a hyperplane, etc.).  

Note that in the discussion above we may in fact take $T = {\rm tube}(\delta,\mathcal{M})$. In that case \eqref{equ:DisttoMan} will hold simultaneously for all ${\bf z} \in \mathbbm{R}^N$ and $({\bf t}, \delta) \in \mathbbm{R}^N \times \mathbbm{R}^+$ with ${\bf t} \in {\rm tube}(\delta,\mathcal{M})$ so that $f:  \mathbbm{R}^N \rightarrow \mathbbm{R}^m$ will approximately preserve the distances of all points ${\bf z} \in \mathbbm{R}^N$ to ${\rm tube}(\delta,\mathcal{M})$ up to errors on the order of $\mathcal{O}(\varepsilon) d({\bf z}, {\rm tube}(\delta,\mathcal{M})) + \mathcal{O}(\delta)$ for all $\delta \in \mathbbm{R}^+$.  This is in fact rather remarkable when one recalls that the best achievable embedding dimension, $m$, here only depends on the geometric properties of the low-dimensional manifold $\mathcal{M}$ (see Corollary~\ref{cor:main} for a detailed accounting of these dependences).

\subsection{The Main Result and Applications}

Let $X$ be an arbitrary subset of $\mathbbm{R}^N$. The \textbf{Gaussian width} of $X$ is  
\begin{align*}
	w(X) := \mathbb{E} \sup_{{\bf x} \in X} \, \langle {\bf g},{\bf x} \rangle,
\end{align*}
where ${\bf g}$ is a random vector with $N$ independent entries following a standard normal distribution. The Gaussian width is considered one of the basic geometric quantities associated with subsets $X$ of $\mathbbm{R}^N$, such as volume and surface area. Indeed, this quantity plays a central role in high-dimensional probability and its applications. For a nice introduction to the notion of Gaussian width we refer the interested reader to \cite[Chapter 7]{vershynin2018high-dimensional}. We further define the \textbf{unit secants} of $X$ to be the set
\begin{equation}\label{eq: unit secants}
	 S_X := \overline{\left \{ \frac{\x-\y}{\|\x-\y\|_2} \colon \x \neq \y \in X\right \}}.
	\end{equation}
Our main result allows its users to embed an arbitrary subset $X\subseteq \mathbbm{R}^N$ into a lower dimensional space $\mathbbm{R}^m$ while not only (approximately) preserving the geometric structure of $X$, but while also (approximately) preserving the distances of all points in $\mathbbm{R}^N \setminus X$ to $X$. In addition, the target dimension $m$ obtained in the process is optimal up to constants.

\begin{Theorem}\label{theo: main}
	Let $X \subseteq \mathbbm{R}^N$. For any $\varepsilon \in (0,1)$ there exists an $f \colon \mathbbm{R}^N \longrightarrow \mathbbm{R}^{m}$ with $m=\mathcal{O}(\varepsilon^{-2} \omega(S_X)^2)$ such that
	\begin{equation}\label{eq: theo main} 
		(1-\varepsilon)\|\x-\y\|_2 \leq \|f(\x)-f(\y)\|_2 \leq (1+\varepsilon)\|\x-\y\|_2 \quad \forall \, \x \in X \quad \forall \, \y \in \mathbbm{R}^N.
	\end{equation}
\end{Theorem}

\begin{remark}
	The proof of Theorem \ref{theo: main} is constructive and provides an algorithm to generate terminal embeddings. For $\varepsilon \in (0,1)$, let $\Phi \in \mathbbm{R}^{m \times N}$ be a linear embedding providing $\frac{\varepsilon}{60}$-convex hull distortion (see (\ref{eq: convex hull distortion}). Observe that Corollary \ref{coro:subgaussiandistorion} guarantees the existence of such $\Phi$ if $m$ is large enough. Additionally, let $g \colon \mathbbm{R}^N \longrightarrow \mathbbm{R}^m$ be a function so that $g(\y)$ satisfies conditions (1) and (2) in Lemma \ref{lem: u'} for $\frac{\varepsilon}{10}$ for all $\y \in \mathbbm{R}^N$. Finally, given $\y \in \mathbbm{R}^N$ let $\y_{\overline{X}}$ be a point in $\overline{X}$ minimizing the distance between $\y$ and $\overline{X}$. Then the proof of Theorem \ref{theo: main} shows that the function $f \colon \mathbbm{R}^N \longrightarrow \mathbbm{R}^{m+1}$ given by
	\[ f(\y)=\left (\Phi\y_{\overline{X}}+g(\y), \sqrt{\|\y-\y_{\overline{X}}\|_2^2 - \|g(\y)\|_2^2}\right ) \quad \forall \, \y \in \mathbbm{R}^N\]
	provides a terminal embedding for $X$ satisfying (\ref{theo: main}).
\end{remark}

Bounds for the Gaussian width of numerous set families can be found in the literature. For instance, when $X \subseteq \mathbbm{R}^N$ is a finite set with $n$ elements, then $\omega(X) \leq c\sqrt{\log n} \operatorname{diam}(X)$, where $c$ is a universal constant and $\operatorname{diam}(X)=\sup_{\x,\y \in X} \|\x-\y\|$ (see Exercise 7.5.10 in \cite{vershynin2018high-dimensional}). Consequently, Theorem \ref{theo: main} provides dimensionality reduction for finite sets of the same order as Nayaran and Nelson's result, which is optimal as previously noted. Thus, the dimensionality reduction provided by Theorem \ref{theo: main} is optimal as well. A more interesting consequence of Theorem \ref{theo: main} appears when we consider $d$-dimensional submanifolds of $\mathbbm{R}^N$. Recall that given a set $X \subseteq \mathbbm{R}^N$, the reach of $X$ measures how far away points can be from $X$ while still having a unique closest point in $X$ \cite{federer1959curvature}. Formally, the \textbf{reach} of $X$ is defined as
\[
\tau_X := \sup \left\{ t \geq 0  ~\big|~ \, \forall \, {\bf x} \in \mathbbm{R}^N \text{ such that } d({\bf x},X) < t, \, {\bf x} \text{ has a unique closest point in } \overline{X} \right\}.
\]
The following result is a restatement of Theorem 4.5 in \cite{iwen2024on}. It bounds the Gaussian width of the unit secants of a smooth submanifold of $\mathbbm{R}^N$ in terms of the manifold's dimension, reach, and volume. 

\begin{Lemma}\label{GaussianWidthOfManifodWithBoundaryViaGunther}
	Let $\mathcal{M}$ be a compact $d$-dimensional submanifold of $\mathbbm{R}^N$ 
	with boundary $\partial \mathcal{M}$, finite reach $\tau_{\mathcal{M}}$, and volume  $V_{\mathcal M}$.  Enumerate the connected components of $\partial \mathcal{M}$ and let $\tau_i$ be the reach of the $i^{\rm th}$ connected component of $\partial \mathcal M$ as a submanifold of $\mathbbm{R}^N$. 
	Set $\tau := \min_{i} \{\tau_{\mathcal M}, \tau_i \}$, let $V_{\partial \mathcal{M}}$ be the volume of $\partial \mathcal{M}$, and denote the volume of the $d$-dimensional Euclidean ball of radius $1$ by $\omega_d$.  Next, 
	\begin{enumerate}
		\item if $d=1$, define $\alpha_{\mathcal M} :=  \frac{20 V_{\mathcal M}}{\tau}  + V_{\partial {\mathcal M}}$, else
		\item if $d \geq 2$, define $\alpha_{\mathcal M} :=
		\frac{V_\mathcal{M}}{ \omega_d} \left(\frac{41}{\tau} \right)^d 
		+ \frac{V_{\partial \mathcal{M}}}{ \omega_{d-1}} \left(\frac{81}{\tau} \right)^{d-1}$.  
	\end{enumerate}
	Finally, define 
	\begin{align}
		\beta_{\mathcal{M}} &:= \left(\alpha_{\mathcal M}^2 +3^{d} \alpha_{\mathcal M} \right). \label{equ:betadef}
	\end{align}
	Then, the Gaussian width of the unit secants of $\mathcal{M}$ satisfies $$w \left(  S_{\mathcal M} \right)  \leq 8\sqrt{2}\sqrt{\log \left(\beta_{\mathcal{M}} \right)+4d}.$$
\end{Lemma}

In view of the previous bound, the following result is a direct consequence of Theorem \ref{theo: main}

\begin{Corollary}\label{cor:main}
	Let $\mathcal{M}$ be a compact $d$-dimensional submanifold of $\mathbbm{R}^N$ 
	with boundary $\partial \mathcal{M}$, finite reach $\tau_{\mathcal{M}}$, and volume  $V_{\mathcal M}$.  Enumerate the connected components of $\partial \mathcal{M}$ and let $\tau_i$ be the reach of the $i^{\rm th}$ connected component of $\partial \mathcal M$ as a submanifold of $\mathbbm{R}^N$. 
	Set $\tau := \min_{i} \{\tau_{\mathcal M}, \tau_i \}$, let $V_{\partial \mathcal{M}}$ be the volume of $\partial \mathcal{M}$, and denote the volume of the $d$-dimensional Euclidean ball of radius $1$ by $\omega_d$. Set $\beta_{\mathcal{M}}$ as in (\ref{equ:betadef}). Then, for any $\varepsilon \in (0,1)$ there exists $f \colon \mathbbm{R}^N \longrightarrow \mathbbm{R}^m$ with $m=\mathcal{O}(\varepsilon^{-2}(\log (\beta_{\mathcal{M}}) + d))$ such that
	\[ (1-\varepsilon)\|\x-\y\|_2 \leq \|f(\x)-f(\y)\|_2 \leq (1+\varepsilon)\|\x-\y\|_2 \quad \forall \, \x \in \mathcal{M} \quad \forall \, \y \in \mathbbm{R}^N.\]
\end{Corollary}

We further note that alternate applications of Theorem~\ref{theo: main} involving other data models are also possible.  Suppose that $\mathcal{M}$ is a union of $n$ $d$-dimensional affine subspaces so that its unit secants, $S_{\mathcal{M}}$ defined as per (\ref{eq: unit secants}), are contained in the union of at most ${n \choose 2} + n$ unit spheres $\subseteq \mathbb{S}^{N-1}$, each of dimension at most $2d+1$.  The Gaussian width of $S_{\mathcal{M}}$ can then be upper-bounded by $C \sqrt{d+\log n}$ using standard techniques, where $C \in \mathbbm{R}^+$ is an absolute constant.  
An application of Theorem~\ref{theo: main} now guarantees the existence of a terminal embedding $f: \mathbbm{R}^N \rightarrow \mathbbm{R}^{\mathcal{O}\left(\frac{d + \log n}{\varepsilon^2}\right)}$ which will allow approximate nearest subspace queries to be answered for any input point ${\bf z} \in \mathbbm{R}^N$ using only $f({\bf z})$ in the compressed $\mathcal{O}\left(\frac{d + \log n}{\varepsilon^2}\right)$-dimensional space.  Even more specifically, if we choose, e.g., $\mathcal{M}$ to consist of all at most $s$-sparse vectors in $\mathbbm{R}^N$ (i.e., so that $\mathcal{M}$ is the union of $n = {N \choose s}$ subspaces of $\mathbbm{R}^N$), we can now see that Theorem~\ref{theo: main} guarantees the existence of a deterministic compressed estimator \eqref{equ:CompressedEstimator} which allows for the accurate approximation of the best $s$-term approximation error $\displaystyle \inf_{{\bf y} \in \mathbbm{R}^N ~\textrm{at most}~s~\textrm{sparse}} \| {\bf z} - {\bf y} \|_2$ for all ${\bf z} \in \mathbbm{R}^N$ using only $f({\bf z}) \in \mathbbm{R}^{\mathcal{O}(s \log (N/s))}$ as input.  Note that this is only possible due to the non-linearity of $f$ herein.  In, e.g., the setting of classical compressive sensing theory where $f$ must be linear it is known that such good performance is impossible \cite[Section 5]{cohen2009compressed}.

The remainder of the paper is organized as follows.  In Section~\ref{sec:Prelims} we review notation and introduce some definitions. Next, in Section~\ref{sec:ExtensionRes} we provide the proof for Theorem \ref{theo: main} and Corollary \ref{cor:main}. Finally, in Section~\ref{sec:Numerics} we conclude by demonstrating that terminal embeddings allow for more accurate compressive nearest neighbor classification than standard linear embeddings in practice.

\section{Notation and Preliminaries}
\label{sec:Prelims}

In this paper we exclusively work with the Euclidean norm in $\mathbbm{R}^N$, which we denote $\|\cdot\|_2$. The {\bf closure} of a set $X\subseteq \mathbbm{R}^N$ is denoted by $\overline{X}$, and its {\bf convex hull} by $\textrm{conv}(X)$. We will denote the cardinality of a finite set $X$ by $|X|$. The {\bf radius} and {\bf diameter} of $X$ are
\[{\rm rad}(X) := \sup_{{\bf x} \in X} \| {\bf x} \|_2 \quad \mbox{and} \quad {\rm diam}(X) := \sup_{{\bf x}, {\bf y} \in X} \| {\bf x} - {\bf y} \|_2,\]
respectively. Given  $\delta>0$, a {\bf $\delta$-cover} of $X$ will be a subset $S \subseteq X$ such that the following holds:
\[\forall \, {\bf x} \in X ~\exists \, {\bf y} \in S ~{\rm such~that~} \| {\bf x} - {\bf y} \|_2 \leq \delta.\]
Given $\varepsilon>0$, we say that a matrix $\Phi \in \mathbbm{R}^{m \times N}$ is an  \textbf{$\varepsilon$-JL embedding} of a set $X \subseteq \mathbbm{R}^N$ into $\mathbbm{R}^m$ if 
\[ (1-\varepsilon)\|\x-\y\|_2 \leq \|\Phi\x-\Phi\y\|_2 \leq (1+\varepsilon)\|\x-\y\|_2 \quad \forall\, \x,\y \in X.\]
  Here we will be working with random matrices which will embed sets $X$ of bounded size (measured with respect to, e.g., Gaussian Width \cite{vershynin2018high-dimensional}) with high probability. Such matrix distributions are often called {\bf oblivious} and discussed as randomized embeddings in the absence of any specific set $X$ since their embedding quality can be determined independently of any properties of a given set $X$ beyond its size.  In particular, the class of random matrices having independent, isotropic, and sub-Gaussian rows will receive special attention below. A random vector $Y \in \mathbbm{R}^N$ is said to be \textbf{isotropic} if $\mathbb{E} YY^T=I_N$, where $I_N$ denotes the identity matrix in $\mathbbm{R}^N$. This condition is the high-dimensional equivalent of a random variable having zero mean and unit variance. A random variable $Y$ is called \textbf{sub-Gaussian} if there is a constant $K>0$ so that 
  \[\mathbb{P} \{ |Y|\geq t\} \leq 2e^{-t^2/K} \quad \forall \, t\geq 0.\] 
  If $Y\in \mathbbm{R}^N$ is a random vector, we call it sub-Gaussian if its one-dimensional marginals are sub-Gaussian random variables.
\section{The Main Bi-Lipschitz Extension Results and Their Proofs}
\label{sec:ExtensionRes}

In this section we present the proof of Theorem \ref{theo: main} and obtain Corollary \ref{cor:main} as a consequence. The strategy to construct terminal embeddings will be the following: First, we find a convenient Johnson-Lindenstrauss embedding $\Phi \colon X \subseteq \mathbbm{R}^N \longrightarrow \mathbbm{R}^m$ with small distortion. Next, we will consider $f \colon \mathbbm{R}^N \longrightarrow \mathbbm{R}^{m+1}$ to be an outer extension of $\Phi$ given by
\begin{equation}\label{eq: formula terminal}
	 f(\y)=\left (\Phi\y_{\overline{X}} + \y', \sqrt{\|\y-\y_{\overline{X}}\|_2^2 - \|\y'\|_2^2}\right ) \quad \forall \, \y \in \mathbbm{R}^N,
\end{equation}
where $\y_{\overline{X}}$ is a point in $\overline{X}$ with minimal distance from $\y$, that is, $\y_{\overline{X}}= \text{argmin}_{\x \in \overline{X}} \; \| \y-\x\|_2$, and $\y' = g(\y)$ is a point in $\mathbbm{R}^m$ satisfying additional geometric properties that guarantee that $f$ is a terminal embedding. Here, the extra final coordinate is used to make up for error introduced by $\y' \in \mathbbm{R}^m$.  We begin the proof in the next subsection by considering the matrices $\Phi$ whose outer extensions will eventually yield our terminal embeddings.

\subsection{Sub-Gaussian Matrices and $\varepsilon$-Convex Hull Distortion for Infinite Sets}

The first step towards proving Theorem \ref{theo: main} is to find a linear JL embedding $\Phi$ of $X$ with small distortion that will later be extended to produce a terminal embedding of the form (\ref{eq: formula terminal}). Toward that end we will use the following slightly stronger property:  Let $S \subseteq \mathbb{S}^{N-1}$ be a subset of the unit sphere in $\mathbbm{R}^N$, and $\varepsilon \in (0,1)$. A matrix $\Phi \in \mathbbm{R}^{m\times N}$ is said to provide \textbf{$\varepsilon$-convex hull distortion} for $S$ if
\begin{equation}\label{eq: convex hull distortion} 
	|\|\Phi\x\|_2-\|\x\|_2|<\varepsilon \quad \forall \, \x \in \operatorname{conv}(S).\footnote{Note that any matrix $\Phi$ that satisfies this condition for $S = S_{X}$ will also be an $\varepsilon$-JL embedding of $X$.}
\end{equation}

This family of embeddings was used in \cite{narayanan2019optimal} to construct terminal embeddings that achieve optimal dimensionality reduction for finite sets. In this section, we use results from \cite{vershynin2018high-dimensional} to study the existence of convex hull distortion embeddings for infinite sets. Our main tool will be the following result (see also \cite[Theorem 4]{iwen2024on}).

\begin{Theorem}[See Theorem 9.1.1 and Exercise 9.1.8 in \cite{vershynin2018high-dimensional}]
	\label{vershynin-matrix-deviation-theorem}
	Let $\Phi \in \mathbbm{R}^{m\times N}$ be a matrix whose rows are independent, isotropic, and sub-Gaussian random vectors in $\mathbbm{R}^N$. Consider a subset $X \subseteq \mathbbm{R}^N$ and let $p \in (0,1)$. Then there exists a constant $c$ depending only on the distribution of the rows of $\Phi$ such that
	\begin{align*}
		\sup_{{\bf x} \in X} \left| \| \Phi {\bf x} \|_2 - \sqrt{m} \| {\bf x} \|_2   \right| 
		\leq 
		c \left[w(X) + \sqrt{\ln(2/p)} \cdot {\rm rad}(X) \right]
	\end{align*}
	holds with probability at least $1 - p$.
\end{Theorem}

Observe that random matrices satisfying the assumptions of Theorem \ref{vershynin-matrix-deviation-theorem} include those with independent rows following a multivariate standard normal distribution, that is, matrices whose rows are of the form ${\bf g}=(g_1,\ldots,g_N)$ with the $g_i$ being independent standard normal random variables. The main result of this section is a direct consequence of Theorem~\ref{vershynin-matrix-deviation-theorem} together with standard results concerning Gaussian widths \cite[Proposition 7.5.2]{vershynin2018high-dimensional}.

\begin{Corollary}
	\label{coro:subgaussiandistorion}
	Let $\Phi \in \mathbbm{R}^{m\times N}$ be a matrix whose rows are independent, isotropic, and sub-Gaussian random vectors in $\mathbbm{R}^N$. Consider a subset $X \subseteq \mathbbm{R}^N$ and let $p \in (0,1)$. Given $\varepsilon \in (0,1)$, suppose that
	$$m \geq \frac{c'}{\varepsilon^2} \left( w \left(S_{X} \right) + \sqrt{\ln(2/p)} \right)^2,$$
	where $c'$ is a constant depending only on the distribution of the rows of $\Phi$. Then, with probability at least $1 - p$ the random matrix $\frac{1}{\sqrt{m}} \Phi$ provides $\varepsilon$-convex hull distortion for $S_{X}$.
\end{Corollary}

\begin{proof}
	Consider $S = \textrm{conv}\left(S_{X} \right)$. Note that $w \left(\textrm{conv}\left(S_{X}\right) \right) = w \left(S_{X} \right)$ \cite[Proposition 7.5.2]{vershynin2018high-dimensional}, and that ${\rm rad}(S)\leq 1$. Therefore, the result follows with $c'=c^2$ from applying Theorem~\ref{vershynin-matrix-deviation-theorem} to $S$, where $c$ is the constant that appears in the statement of Theorem \ref{vershynin-matrix-deviation-theorem}.
\end{proof}

\subsection{Outer Extensions of $\varepsilon$-Convex Hull Embeddings}

Let $X$ be a subset of $\mathbbm{R}^N$ and let $\Phi$ provide $\varepsilon$-convex hull distortion for $S_X$. Consider $f$ an outer extension of $\Phi$ of the form (\ref{eq: formula terminal}). The final step to prove Theorem \ref{theo: main} is to find a mapping $\y \in \mathbbm{R}^N \rightarrow \y' \in \mathbbm{R}^m$ that makes of this outer extension a terminal embedding. For a fixed $\y \in \mathbbm{R}^N \setminus \overline{X}$ and $\x \in X\setminus \{\y_{\overline{X}}\}$, let $\alpha_\x$ denote the angle between the vectors $\y-\y_{\overline{X}}$ and $\x-\y_{\overline{X}}$. Intuitively, $\y' \in \mathbbm{R}^m$ is a point chosen so that $\alpha_\x$ is very close to the angle between $\y'$ and $\Phi(\x-\y_{\overline{X}})$, for all $\x \in X\setminus\{\y_{\overline{X}}\}$. This property allows us to effectively extend $\Phi$ to $\y$ while preserving angles, which leads to preservation of the norm. The existence of a point satisfying such a property is far from trivial, and it was proved in Lemma 3.1 in \cite{mahabadi2018nonlinear} for finite sets. The following lemma is an optimized version of their result from \cite{narayanan2019optimal}. 

\begin{Lemma}[Lemma 3.6 in \cite{narayanan2019optimal}]\label{lem: narayan u'}
	Let $\x_1,\ldots,\x_n$ be nonzero points in $\mathbbm{R}^N$. Suppose that $\Phi \in \mathbbm{R}^{m \times N}$ provides $\varepsilon$-convex hull distortion for $V=\left \{\pm \frac{\x_i}{\|\x_i\|_2}\colon i=1,\ldots,n\right \}$. Then, for any $\y \in \mathbbm{R}^N$ there is $\y' \in \mathbbm{R}^m$ such that $\|\y'\|_2\leq \|\y\|_2$ and $|\langle \y',\Phi \x_i\rangle - \langle \y,\x_i\rangle|\leq \varepsilon \|\y\|_2\|\x_i\|_2$ for every $\x_i$.
\end{Lemma}
The von Neumann minimax theorem \cite{neumann1928zur} plays a crucial role in the proof of the previous result. The following result is an extended version of Lemma \ref{lem: narayan u'} that provides existence of the point $\y'$ for infinite sets.

\begin{Lemma}\label{lem: u'}
    Let $X$ be a subset of $\mathbbm{R}^N$. For $\y \in \mathbbm{R}^N$, let $\y_{\overline{X}}= \text{argmin}_{\x \in \overline{X}} \; \| \y-\x\|_2$, and define $W_\y=\left \{\pm \frac{\x-\y_{\overline{X}}}{\|\x-\y_{\overline{X}}\|_2} \colon \x \in X \setminus\{\y_{\overline{X}}\}\right \}\subseteq S_X$. Suppose that $\Phi \in \mathbbm{R}^{m\times N}$ provides $\frac{\varepsilon}{6}$-convex hull distortion for $W_\y$. Then, there is $\y'\in \mathbbm{R}^m$ such that
	\begin{enumerate}
		\item $\|\y'\|_2\leq \|\y-\y_{\overline{X}}\|_2$;
		\item $|\langle \y',\Phi(\x-\y_{\overline{X}})\rangle - \langle \y-\y_{\overline{X}}, \x-\y_{\overline{X}}\rangle | \leq \varepsilon \|\y-\y_{\overline{X}}\|_2\|\x-\y_{\overline{X}}\|_2$ for all $\x \in X$.
	\end{enumerate}
\end{Lemma}

\begin{proof}
	If $\y \in \overline{X}$, then one can take $\y'=0$. Thus, we may assume that $\y \in \mathbbm{R}^N \setminus \overline{X}$.
	If $X$ is finite the result follows from Lemma \ref{lem: narayan u'}. For infinite sets, let $\mathcal{C}\subseteq W_\y$ be a finite $\frac{\varepsilon}{4}$-cover of $W_\y$. Notice that the existence of such a covering is guaranteed since $\overline{W_\y}$ is compact. Then, Lemma \ref{lem: narayan u'} provides $\y' \in \mathbbm{R}^m$ with $\|\y'\|_2\leq \|\y-\y_{\overline{X}}\|_2$ such that
	\[ |\langle \y',\Phi\w\rangle - \langle \y-\y_{\overline{X}}, \w \rangle | \leq \frac{\varepsilon}{6} \|\y-\y_{\overline{X}}\|_2 \quad \forall \, \w \in \mathcal{C}.\]
	For any $\x \in X$, write $\tilde{\x}=\frac{\x-\y_{\overline{X}}}{\|\x-\y_{\overline{X}}\|_2}$ and let $\w \in \mathcal{C}$ so that $\|\tilde{\x}-\w\|_2\leq\frac{\varepsilon}{4}$. Observe that
	\begin{align*}
		\left| \langle \y', \Phi\tilde{\x} \rangle - \langle \y-\y_{\overline{X}}, \tilde{\x} \rangle \right| &\leq \left| \langle \y', \Phi(\tilde{\x}-\w) \rangle \right| + \left| \langle \y', \Phi\w \rangle - \langle \y-\y_{\overline{X}}, \w  \rangle \right| + \left| \langle \y-\y_{\overline{X}}, \tilde{\x}-\w \rangle \right|\\
		&\leq \|\y-\y_{\overline{X}}\|_2 \left ( \|\Phi ( \tilde{\x}-\w)\|_2 + \frac{\varepsilon}{6} + \|\tilde{\x}-\w\|_2\right ) \\
		&= \|\y-\y_{\overline{X}}\|_2\left (2\left \|\Phi\left (\frac{\tilde{\x}-\w}{2}\right )\right \|_2  + \frac{\varepsilon}{6} + \|\tilde{\x}-\w\|_2 \right)\\
		&\leq \|\y-\y_{\overline{X}}\|_2 \left (\frac{\varepsilon}{2} + 2\|\tilde{\x}-\w\|_2\right )\leq \varepsilon \|\y-\y_{\overline{X}}\|_2,
	\end{align*}
	where the first inequality follows from the triangle inequality, the second one from Cauchy–Schwarz inequality and $\| \y' \|_2\leq \| \y-\y_{\overline{X}}\|_2$, and the forth one from the $\frac{\varepsilon}{6}$-convex hull distortion condition of $\Phi$.
\end{proof}

\subsection{Proof of Main Results}

Let $X$ be a subset of $\mathbbm{R}^N$ and suppose that $\Phi \in \mathbbm{R}^{m\times N}$ provides convex hull distortion for $S_X$. Consider a map $g\colon \mathbbm{R}^N \longrightarrow \mathbbm{R}^m$ that assigns to every $\y \in \mathbbm{R}^N$ a point $g(\y)=\y'$ satisfying conditions (1) and (2) in Lemma \ref{lem: u'}. We are now ready to prove that if $f\colon \mathbbm{R}^N\longrightarrow \mathbbm{R}^{m+1}$ is an outer extension of $\Phi$ defined as in (\ref{eq: formula terminal}), then $f$ is an embedding satisfying the terminal condition (\ref{eq: terminal embedding}).

\begin{proof}[Proof of Theorem \ref{theo: main}]
Given $\y \in \R^N$ let $\y_{\overline{X}} \in \overline{X}$ satisfying $\left\| \y - \y_{\overline{X}} \right\|_2 = \min_{\x \in \overline{X}} \left\| \y - \x \right\|_2$. Let $\Phi \in \mathbbm{R}^{m\times N}$ be an $\frac{\varepsilon}{60}$-convex hull distortion for $S_X$. In view of Lemma \ref{lem: u'}, there is a map $g\colon \mathbbm{R}^N \longrightarrow \mathbbm{R}^m$ satisfying $\|g(\y)\|_2\leq \|\y-\y_{\overline{X}}\|_2$ and
\begin{equation}\label{eq: g condition}
	 |\langle g(\y),\Phi(\x-\y_{\overline{X}})\rangle - \langle \y-\y_{\overline{X}}, \x-\y_{\overline{X}}\rangle | \leq \frac{\varepsilon}{10} \|\y-\y_{\overline{X}}\|_2\|\x-\y_{\overline{X}}\|_2 \quad \forall \, \x \in X \quad \forall \, \y \in \mathbbm{R}^N.
\end{equation}
We define $f \colon \mathbbm{R}^N \longrightarrow \mathbbm{R}^{m+1}$ by
	\[
	f(\y) := \begin{cases}
	\left( \Phi \y, 0 \right) & \text{if } \y \in \overline{X}; \\
	\left( \Phi \y_{\overline{X}} + g(\y ), \sqrt{\left\| \y - \y_{\overline{X}} \right\|_2^2 - \left\| g(\y) \right\|_2^2} \right) & \text{if } \y \notin \overline{X}.
	\end{cases}
 	\]
Fix $\x \in X$. Observe that if $\y \in \overline{X}$ then $\| f(\x) - f(\y) \|_2 = \| \Phi(\x - \y) \|_2$. Since $\Phi$ is an $\frac{\varepsilon}{60}$-convex hull distortion for $S_X$, we deduce that
\[ \left | \|\Phi (\x-\y)\|_2 - \|\x - \y \|_2 \right | \leq \frac{\varepsilon}{60} \|\x-\y\|_2.\] 
Thus, it suffices to consider a fixed $\y \notin \overline{X}$.  In that case we have
\begin{align}
    \| f(\x) - f(\y) \|_2^2 &= \| \Phi(\x - \y_{\overline{X}}) - g(\y)\|_2^2 + \left\| \y - \y_{\overline{X}} \right\|_2^2 - \left\| g(\y) \right\|_2^2 \nonumber \\
    &= \left\| \y - \y_{\overline{X}} \right\|_2^2 + \| \Phi(\x - \y_{\overline{X}}) \|_2^2 - 2  \langle g(\y), \Phi(\x - \y_{\overline{X}}) \rangle 
    \label{equ:finalextendeqproof1}
\end{align}
by the polarization identity and parallelogram law. Similarly we have that  
\begin{equation}
\label{equ:finalextendeqproof2}
    \| \x - \y \|_2^2 = \left\| \left(\x - \y_{\overline{X}} \right) - \left(\y - \y_{\overline{X}} \right) \right\|_2^2 =\| \y - \y_{\overline{X}} \|_2^2 + \| \x - \y_{\overline{X}} \|_2^2 - 2 \langle \y - \y_{\overline{X}}, \x - \y_{\overline{X}} \rangle.
\end{equation}
Subtracting \eqref{equ:finalextendeqproof2} from \eqref{equ:finalextendeqproof1} we can now see that
\[\left| \| f(\x) - f(\y) \|_2^2 - \| \x - \y \|_2^2 \right| \leq \left| \| \Phi(\x - \y_{\overline{X}}) \|_2^2 - \| \x - \y_{\overline{X}} \|_2^2 \right| +
    2\left|
     \langle g(\y), \Phi(\x - \y_{\overline{X}}) \rangle  - \langle \y - \y_{\overline{X}}, \x - \y_{\overline{X}} \rangle \right|.
\]
Observe that
\begin{align*} 
	\left |\| \Phi(\x - \y_{\overline{X}}) \|_2^2 - \| \x - \y_{\overline{X}} \|_2^2\right | &= \left |\| \Phi(\x - \y_{\overline{X}}) \|_2 - \| \x - \y_{\overline{X}} \|_2\right | \cdot \left |\| \Phi(\x - \y_{\overline{X}}) \|_2 + \| \x - \y_{\overline{X}} \|_2\right | \\
	&\leq \frac{\varepsilon}{60}\|\x-\y_{\overline{X}}\|_2 \cdot  \left |2\|\x - \y_{\overline{X}} \|_2 + \| \x - \y_{\overline{X}} \|_2\right | \leq \frac{\varepsilon}{20} \|\x-\y_{\overline{X}}\|_2^2.
\end{align*}
Therefore, in view of (\ref{eq: g condition}) we obtain that
\[
	\left| \| f(\x) - f(\y) \|_2^2 - \| \x - \y \|_2^2 \right| \leq \frac{\varepsilon}{20} \|\x-\y_{\overline{X}}\|_2^2 + \frac{\varepsilon}{5} \|\y-\y_{\overline{X}}\|_2\|\x-\y_{\overline{X}}\|_2 \leq \frac{\varepsilon}{4} \|\x-\y_{\overline{X}}\|_2^2\leq \varepsilon \|\x-\y\|_2^2.
\]
since $\| \x - \y_{\overline{X}} \|_2 \leq \| \x - \y \|_2 + \| \y - \y_{\overline{X}} \|_2 \leq 2 \| \x - \y \|_2$. The result follows from observing that
\[ \left| \| f(\x) - f(\y) \|_2 - \| \x - \y \|_2 \right| = \frac{\left| \| f(\x) - f(\y) \|_2^2 - \| \x - \y \|_2^2 \right|}{\left| \| f(\x) - f(\y) \|_2 + \| \x - \y \|_2 \right|} \leq \frac{\varepsilon\|\x-\y\|_2^2}{\|\x-\y\|_2} = \varepsilon \|\x-\y\|_2.\qedhere\]

\end{proof}

\begin{proof}[Proof of Corollary \ref{cor:main}]
	It follows from Theorem \ref{theo: main} and the bounds for the Gaussian width provided by Lemma \ref{GaussianWidthOfManifodWithBoundaryViaGunther}.
\end{proof}

\section{A Numerical Evaluation of Terminal Embeddings}
\label{sec:Numerics}

\renewcommand{\algorithmicrequire}{\textbf{Input:}}
\renewcommand{\algorithmicensure}{\textbf{Output:}}


In this section we consider several variants of the optimization approach mentioned in Subsection 3.3 of \cite{narayanan2019optimal} for implementing a terminal embedding $f: \mathbbm{R}^N \rightarrow \mathbbm{R}^{m+1}$ of a finite set $X \subseteq \mathbbm{R}^N$.  In effect, this requires us to implement a function satisfying two sets of constraints from \cite[Subsection 3.3]{narayanan2019optimal} that are analogous to the two properties that $y'\in \mathbbm{R}^m$ in Lemma \ref{lem: u'} satisfies.  See Lines 2 and 3 of Algorithm~\ref{alg:FTE} for a concrete example of one type of constrained minimization problem solved herein to accomplish this task.  

\begin{algorithm}[H]
\caption{Terminal Embedding of a Finite Set} \label{alg:FTE}
\begin{algorithmic}[1]
\Require
$\epsilon \in (0,1),~ X \subseteq \mathbbm{R}^N,~ n:=\lvert X \rvert,~ S \subseteq \mathbbm{R}^N,~ n':=\lvert S\rvert$,~$m \in \mathbbm{N}$ with $m < N$,~a random matrix with i.i.d. standard Gaussian entries, $\Pi \in \mathbbm{R}^{m \times N}$, rescaled to perform as a JL embedding matrix $\Phi := \frac{1}{\sqrt{m}} \Pi$ 
\Ensure A terminal embedding of $X$, $f \in \mathbbm{R}^N \rightarrow \mathbbm{R}^{m+1}$, evaluated on $S$
\For {${\bf y} \in S$}
\State  Compute ${\bf y}_{\overline{X}} := \text{argmin}_{{\bf x} \in X} \; \| {\bf y} - {\bf x}\|_{2}$
\State  Solve the following constrained minimization problem to compute a minimizer ${\bf y}' \in \mathbb{R}^m$
\vspace{-3mm}
\begin{align*}
\text{Minimize} \quad &h_{{\bf y},{\bf y}_{\overline{X}}}({\bf z}) := \| {\bf z}\|^2_2 + 2 \langle \Pi ({\bf y} - {\bf y}_{\overline{X}}),{\bf z} \rangle&\\
\text{subject to} \quad &\|{\bf z}\|_{2} \leq \|{\bf y} - {\bf y}_{\overline{X}}\|_{2}&\\
&\lvert \langle {\bf z}, \Phi ({\bf x} - {\bf y}_{\overline{X}}) \rangle  - \langle {\bf y} - {\bf y}_{\overline{X}}, {\bf x} - {\bf y}_{\overline{X}}\rangle\rvert  \leq  \epsilon \|{\bf y} - {\bf y}_{\overline{X}}\|_{2}\|{\bf x} - {\bf y}_{\overline{X}}\|_{2}~\forall\, {\bf x} \in X&
\end{align*}
\State Compute $f: \mathbbm{R}^N \rightarrow \mathbbm{R}^{m+1}$ at ${\bf y}$ via
\vspace{-4mm}
\begin{align*}
f({\bf y}) :=
\begin{cases}
(\Pi {\bf y},0), \quad &{\bf y} \in X\cr
(\Pi {\bf y}_{\overline{X}} + {\bf y}', \sqrt{\|{\bf y} - {\bf y}_{\overline{X}} \|_{2}^2 - \|{\bf y}'\|_{2}^2}), \quad &{\bf y} \notin X\cr
\end{cases}
\end{align*}
\EndFor
\end{algorithmic}
\end{algorithm}

 Crucially, we note that any point ${\bf y}' \in \mathbb{R}^m$ satisfying the two sets of constraints in Line 3 of Algorithm~\ref{alg:FTE} for a given ${\bf y} \in \mathbbm{R}^N$ is guaranteed to correspond to an evaluation of a valid terminal embedding of $X$ at ${\bf y}$ in Line 4. Furthermore, the choice of $\y'$ can drastically change the performance of the terminal embedding. Given this setup, several heretofore unexplored practical questions about terminal embeddings immediately present themselves.  These include:
\begin{enumerate}

    \item Repeatedly solving the optimization problem in Line 3 of Algorithm~\ref{alg:FTE} to evaluate a terminal embedding of $X$ on $S$ is certainly more computationally expensive than simply evaluating a standard linear Johnson-Lindenstrauss (JL) embedding of $X$ on $S$ instead.  How do terminal embeddings empirically compare to standard linear JL embedding matrices on real-world data in the context of, e.g., compressive classification?  When, if ever, is their additional computational expense actually justified in practice?
    
    \item Though any choice of  objective function $h_{{\bf y},{\bf y}_{\overline{X}}}$ in Line 3 of Algorithm~\ref{alg:FTE} must result in a terminal embedding $f$ of $X$ based on the available theory, some choices probably lead to better empirical performance than others.  What's a good default choice?
    
    \item How much dimensionality reduction are terminal embeddings capable of in the context of, e.g., accurate compressive classification using real-world data?   
    
\end{enumerate}

In keeping with the motivating application discussed in Subsection~\ref{sec:motivating_application} above, we will explore some preliminary answers to these three questions in the context of compressive classification based on real-world data below.  

\subsection{A Comparison Criteria:  Compressive Nearest Neighbor Classification}

Given a labelled data set $\mathcal{D} \subseteq \mathbbm{R}^N$ with label set $\mathcal{L}$, we let $Label: \mathcal{D} \rightarrow \mathcal{L}$ denote the function which assigns the correct label to each element of the data set.  To address the three questions above we will use compressive nearest neighbor classification accuracy as a primary measure of an embedding strategy's quality.  See Algorithm~\ref{Alg:Class} for a detailed description of how this accuracy can be computed for a given data set $\mathcal{D}$.

\begin{algorithm}[H]
\caption{Measuring Compressive Nearest Neighbor Classification Accuracy} \label{Alg:Class}
\begin{algorithmic}[1]
\Require $\epsilon \in (0,1)$, A labeled data set $\mathcal{D} \subseteq \mathbbm{R}^N$ split into two disjoint subsets: A training set $X \subseteq \mathcal{D}$ with $n:=\lvert X \rvert$, and a test set $S \subseteq \mathcal{D}$ with $n':=\lvert S\rvert$, such that $S \cap X = \emptyset$.  A compressive dimension $m < N$.  
\Ensure Successful Nearest Neighbor Classification Percentage for Data Embedded in $\mathbbm{R}^{m+1}$.
\State Fix $f: \mathbbm{R}^N \rightarrow \mathbbm{R}^{m+1}$, an embedding of the training data $X \subseteq \mathbbm{R}^N$ into $\mathbbm{R}^{m+1}$ satisfying
$$(1 - \epsilon) \| {\bf x} - {\bf y} \|_2 \leq \left\| f({\bf x}) - f({\bf y}) \right\|_2 \leq (1 + \epsilon) \| {\bf x} - {\bf y} \|_2$$
for all ${\bf x},{\bf y} \in X$.  [Note: this can either be a JL-embedding of $X$, or a stronger terminal embedding of $X$.]\\

\State \% {\it Embed the training data into $\mathbbm{R}^{m+1}$.}
\For{${\bf x} \in X$}
\State Compute $f({\bf x})$ using, e.g., Algorithm~\ref{alg:FTE}.
\EndFor\\

\State \% {\it Classify the test data using its embedded distance in $\mathbbm{R}^{m+1}$.}
\State $p = 0$ 
\For{${\bf y} \in S$}
\State Compute $f({\bf y})$ using, e.g., Algorithm~\ref{alg:FTE}
\State Compute ${\bf z} = \text{argmin}_{ {\bf x} \in X} \; \|f({\bf y}) -  f({\bf x})\|_{2}$
\If{$Label({\bf y}) = Label({\bf z})$}
\State $p = p + 1$
\EndIf
\EndFor\\

\State Output the Successful Classification Percentage = $\dfrac{p}{n'} \times 100\%$
\end{algorithmic}
\end{algorithm}

Note that Algorithm~\ref{Alg:Class} can be used to help us compare the quality of different embedding strategies.  For example, one can use  Algorithm~\ref{Alg:Class} to compare different choices of objective functions $h_{{\bf y},{\bf y}_{\overline{X}}}$ in Line 3 of Algorithm~\ref{alg:FTE} against one another by running Algorithm~\ref{Alg:Class} multiple times on the same training and test data sets while only varying the implementation of Algorithm~\ref{alg:FTE} each time.  This is exactly the type of approach we will use below.  Of course, before we can begin we must first decide on some labelled data sets $\mathcal{D}$ to use in our classification experiments.

\subsection{Our Choice of Training and Testing Data Sets}
\label{sec:DataSets}

Herein we consider two standard benchmark image data sets which allow for accurate uncompressed Nearest Neighbor (NN) classification.  The images in each data set can then be vectorized and embedded using, e.g., Algorithm~\ref{alg:FTE} in order to test the accuracies of compressed NN classification variants against both one another, as well as against standard uncompressed NN classification.  These benchmark data sets are as follows. 

\begin{figure}[H] 
\centering
\includegraphics[width=0.8\textwidth]{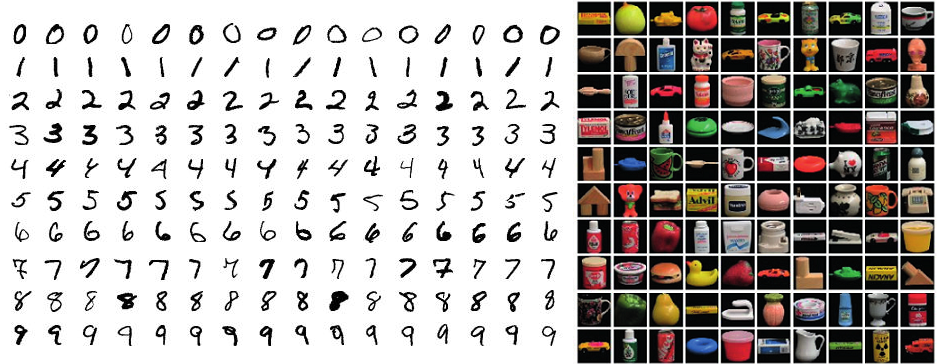}
\caption{Example images from the MNIST data set (left), and the COIL-100 data set (right).}
\label{fig:exampdata}
\end{figure}

{\bf The MNIST data set \cite{li2012the}} consists of 60,000 training images of $28 \times 28$-pixel grayscale hand-written images of the digits $0$ through $9$. Thus, MNIST $10$ labels to correctly classify between, and $N = 28^2 = 784$.  For all experiments involving the MNIST dataset $n / 10$ digits of each type are selected uniformly at random to form the training set $X$, for a total of $n$ vectorized training images in $\mathbbm{R}^{784}$.  Then, $100$ digits of each type are randomly selected from those not used for training in order to form the test set $S$, leading to a total of $n' = 1000$ vectorized test images in $\mathbbm{R}^{784}$.  See the left side of Figure~\ref{fig:exampdata} for example MNIST images.\\

{\bf The COIL-100 data set \cite{nene1996columbia}} is a collection of $128 \times 128$-pixel color images of $100$ objects, each photographed $72$ times where the object has been rotated by $5$ degrees each time to get a complete rotation.  However, only the green color channel of each image is used herein for simplicity.  Thus, herein COIL-100 consists of $7,200$ total vectorized images in $\mathbbm{R}^{N}$ with $N = 128^2 = 16,384$, where each image has one of $100$ different labels (72 images per label).  For all experiments involving this COIL-100 data set, $n / 100$ training images are down sampled from each of the $100$ objects' rotational image sequences.  Thus, the training sets each contain $n / 100$ vectorized images of each object, each photographed at rotations of $\approx 36000 / n$ degrees (rounded to multiples of $5$).  The resulting training data sets therefore all consist of $n$ vectorized images in $\mathbbm{R}^{16,384}$.  After forming each training set, $10$ images of each type are then randomly selected from those not used for training in order to form the test set $S$, leading to a total of $n' = 1000$ vectorized test images in $\mathbbm{R}^{16,384}$ per experiment.  See the right side of Figure~\ref{fig:exampdata} for example COIL-100 images.

\subsection{A Comparison of Four Embedding Strategies via NN Classification}

In this section we seek to better understand $(i)$ when terminal embeddings outperform standard JL-embedding matrices in practice with respect to accurate compressive NN classification, $(ii)$ what type of objective functions $h_{{\bf y},{\bf y}_{\overline{X}}}$ in Line 3 of Algorithm~\ref{alg:FTE} perform best in practice when computing a terminal embedding, and $(iii)$ how much dimensionality reduction one can achieve with a terminal embedding without appreciably degrading standard NN classification results in practice.  To gain insight on these three questions we will compare the following four embedding strategies in the context of NN classification.  These strategies begin with the most trivial linear embeddings (i.e., the identity map) and slowly progress toward extremely non-linear terminal embeddings.
\begin{itemize}
    \item[(a)] {\bf Identity:}  We use the data in its original uncompressed form (i.e., we use the trivial embedding $f: \mathbbm{R}^N \rightarrow \mathbbm{R}^N$ defined by $f({\bf y}) = {\bf y}$ in Algorithm~\ref{Alg:Class}).  Here the embedding dimension is always fixed to be $N$.
    
    \item[(b)] {\bf Linear:}  We compressively embed our training data $X$ using a JL embedding.  More specifically, we generate an $m \times N$ random matrix $\Phi$ with i.i.d. standard Gaussian entries and then set
    $f: \mathbbm{R}^N \rightarrow \mathbbm{R}^{m+1}$ to be $f({\bf y}) := \left(\Phi {\bf y},~0 \right)$ in Algorithm~\ref{Alg:Class} for various choices of $m$.  It is then hoped that $f$ will embed the test data $S$ well in addition to the training data $X$.  Note that this embedding choice for $f$ is consistent with Algorithm~\ref{alg:FTE} where one lets $X = X \cup S$ when evaluating Line 4, thereby rendering the minimization problem in Line 3 irrelevant.
    
    \item [(c)]{\bf A Valid Terminal Embedding That is as Linear as Possible:}  To minimize the pointwise difference between the terminal embedding $f$ computed by Algorithm~\ref{alg:FTE} and the linear map defined above in (b), we may choose the objective function in Line 3 of Algorithm~\ref{alg:FTE} to be $h_{{\bf y},{\bf y}_{\overline{X}}}({\bf z}) := \langle \Phi ({\bf y}_{\overline{X}} - {\bf y}),{\bf z} \rangle$.  To see why solving this minimizes the pointwise difference between $f$ and the linear map in (b), let ${\bf y}'$ be such that $\langle \Phi ({\bf y}_{\overline{X}} - {\bf y}),{\bf z} \rangle$ is minimal subject to the constraints in Line 3 of Algorithm~\ref{alg:FTE} when ${\bf z} = {\bf y}'$.  Since ${\bf y}$ and ${\bf y}_{\overline{X}}$ are fixed here, we note that ${\bf z} = {\bf y}'$ will then also minimize
    \begin{align*}
    ~&~\left\| \Phi ({\bf y}_{\overline{X}} - {\bf y}) \right\|_2^2 + 2\langle \Phi ({\bf y}_{\overline{X}} - {\bf y}),{\bf z} \rangle + \| {\bf y} - {\bf y}_{\overline{X}} \|_2^2\\
    =&~\left\| \Phi ({\bf y}_{\overline{X}} - {\bf y}) \right\|_2^2 + \| {\bf z} \|_2^2+ 2\langle \Pi ({\bf y}_{\overline{X}} - {\bf y}),{\bf z} \rangle + \| {\bf y} - {\bf y}_{\overline{X}} \|_2^2 - \| {\bf z} \|_2^2\\
    =&~\left\| \Phi ({\bf y}_{\overline{X}} - {\bf y}) +{\bf z} \right\|_2^2 + \| {\bf y} - {\bf y}_{\overline{X}} \|_2^2 - \| {\bf z} \|_2^2\\
    =&~\left\| \left(\Phi {\bf y}_{\overline{X}} +{\bf z}, \sqrt{\| {\bf y} - {\bf y}_{\overline{X}} \|_2^2 - \| {\bf z} \|_2^2} \right) - (\Phi {\bf y},0) \right\|_2^2,
    \end{align*}
    subject to the desired constraints.  Hence, we can see that choosing ${\bf z} = {\bf y}'$ as above is equivalent to minimizing $\| f({\bf y}) - (\Phi {\bf y},0) \|_2^2$ over all valid choices of terminal embeddings $f$ that satisfy the existing theory.
    
    \item[(d)]{\bf A Terminal Embedding Computed by Algorithm~\ref{alg:FTE} as Presented:}  This terminal embedding is computed using Algorithm~\ref{alg:FTE} exactly as it is formulated above (i.e., with the objective function in Line 3 chosen to be $h_{{\bf y},{\bf y}_{\overline{X}}}({\bf z}) := \| {\bf z}\|^2_2 + 2 \langle \Phi ({\bf y} - {\bf y}_{\overline{X}}),{\bf z} \rangle$).  Note that this choice of objective function was made to encourage non-linearity in the resulting terminal embedding $f$ computed by Algorithm~\ref{alg:FTE}.  To understand our intuition for making this choice of objective function in order to encourage non-linearity in $f$, suppose that $\| {\bf z}\|^2_2 + 2 \langle \Phi ({\bf y} - {\bf y}_{\overline{X}}),{\bf z} \rangle$ is minimal subject to the constraints in Line 3 of Algorithm~\ref{alg:FTE} when ${\bf z} = {\bf y}'$.  Since ${\bf y}$ and ${\bf y}_{\overline{X}}$ are fixed independently of ${\bf z}$ this means that ${\bf z} = {\bf y}'$ then also minimize 
    \begin{align*}
        \|{\bf z} \|^2_2 + 2 \langle \Phi ({\bf y} - {\bf y}_{\overline{X}}),{\bf z} \rangle + \|\Phi({\bf y} - {\bf y}_{\overline{X}})\|^2_2 &= \| {\bf z} + \Phi({\bf y} - {\bf y}_{\overline{X}}) \|^2_2.
    \end{align*}
    Hence, this objection function is encouraging ${\bf y}'$ to be as close to $-\Phi({\bf y} - {\bf y}_{\overline{X}}) ~=~ \Phi({\bf y}_{\overline{X}} - {\bf y})$ as possible subject to satisfying the constraints in Line 3 of Algorithm~\ref{alg:FTE}.  Recalling (c) just above, we can now see that this is exactly encouraging ${\bf y}'$ to be a value for which the objective function we seek to minimize in (c) is relatively large.
\end{itemize}

We are now prepared to empirically compare the four types of embeddings (a) -- (d) on the data sets discussed above in Section~\ref{sec:DataSets}.  To do so, we run Algorithm~\ref{Alg:Class} four times for several different choices of embedding dimension $m$ on each data set below, varying the choice of embedding $f$ between (a), (b), (c), and (d) for each value of $m$.  The successful classification percentage is then plotted as a function of $m$ for each different data set and choice of embedding.  See Figures~\ref{fig:CompareEmbeddings}(a) and~\ref{fig:CompareEmbeddings}(c) for the results.  In addition, to quantify the extent to which the embedding strategies (b) -- (d) above are increasingly nonlinear, we also measure the relative distance between where each training-set embedding $f$ maps points in the test sets versus where its associated linear training-set embedding would map them.  More specifically, for each embedding $f$ and test point ${\bf y} \in S$ we let
\begin{align*}
\text{Nonlinearity}_{f}({\bf y}) = \dfrac{\|f( {\bf y}) - (\Phi {\bf y},0)\|_{2}}{\|(\Phi {\bf y},0)\|_{2}} \times 100\%
\end{align*}
See Figures~\ref{fig:CompareEmbeddings}(b) and~\ref{fig:CompareEmbeddings}(d) for plots of $$\text {Mean}_{{\bf y} \in S} \text{ Nonlinearity}_{f}({\bf y})$$ for each of the embedding strategies (b) -- (d) on the data sets discussed in Section~\ref{sec:DataSets}.

To compute solutions to the minimization problem in Line 3 of Algorithm \ref{alg:FTE} below we used the MATLAB package CVX \cite{cvx,grant2008graph} with the initialization ${\bf z}_0 = \Phi ( {\bf y} - {\bf y}_{\overline{X}})$ and $\epsilon = 0.1$ in the constraints. All simulations were performed using MATLAB R2021b on an Intel desktop with a 2.60GHz i7-10750H CPU and 16GB DDR4 2933MHz memory. All code used to generate the figures below is publicly available at \url{https://github.com/MarkPhilipRoach/TerminalEmbedding}.

\begin{figure}[H] 
\centering
\includegraphics[width=0.7\textwidth]{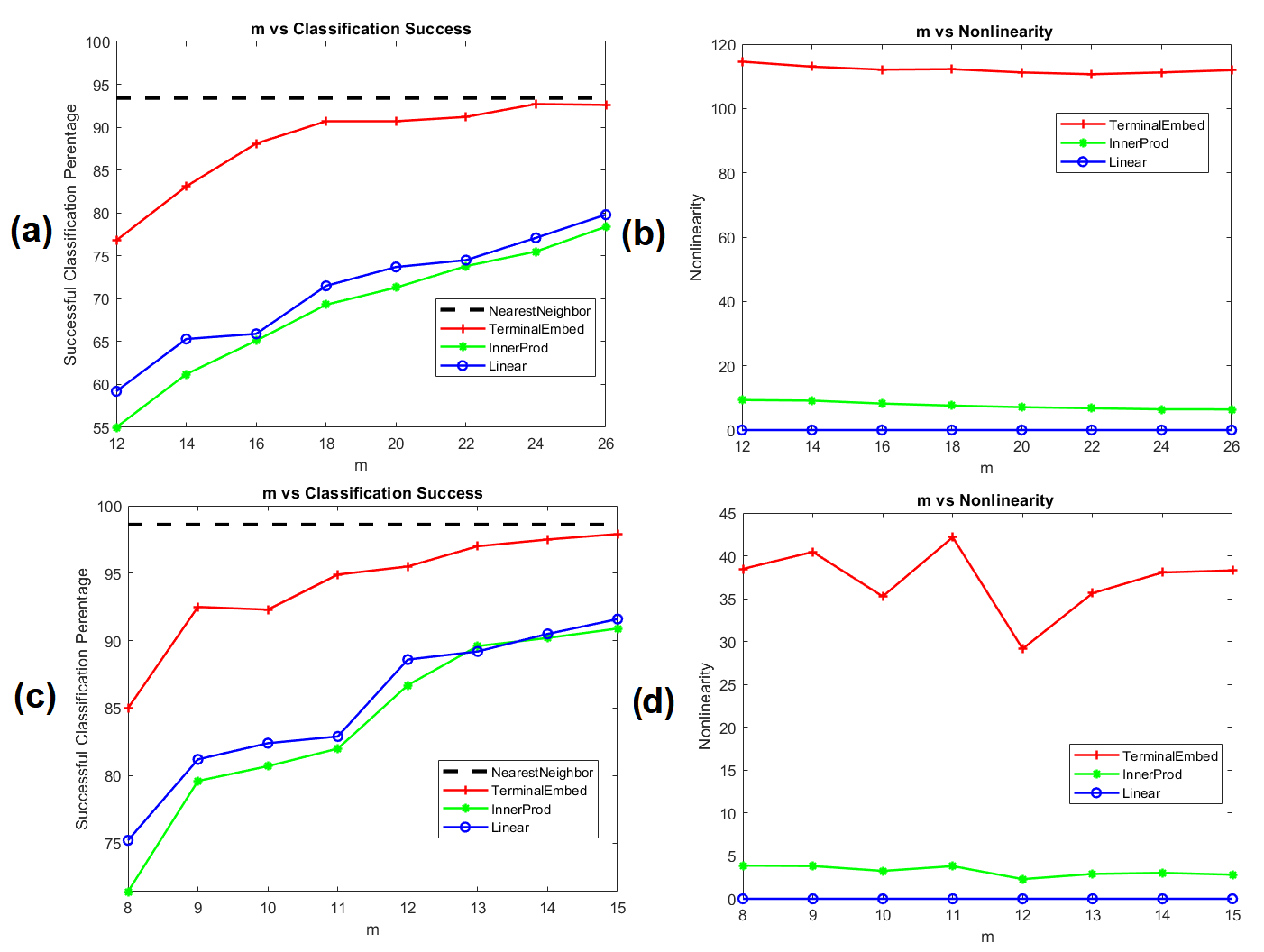}
\caption{Figures \ref{fig:CompareEmbeddings}(a) and \ref{fig:CompareEmbeddings}(b) concern the MNIST data set with training set size $n = 4000$ and test set size $n' = 1000$ in all experiments.  Similarly, Figures \ref{fig:CompareEmbeddings}(c) and \ref{fig:CompareEmbeddings}(d) concern the COIL-100 data set with training set size $n = 3600$ and test set size $n' = 1000$ in all experiments.  In both Figures \ref{fig:CompareEmbeddings}(a) and \ref{fig:CompareEmbeddings}(c) the dashed black ``NearestNeighbor" line plots the classification accuracy when the identity map (a) is used in Algorithm~\ref{Alg:Class}.  Note that the ``NearestNeighbor" line is independent of $m$ because the identity map involves no compression.  Similarly, in all of the Figures \ref{fig:CompareEmbeddings}(a) -- \ref{fig:CompareEmbeddings}(d) the red ``TerminalEmbed" curves correspond to the use of Algorithm~\ref{alg:FTE} as it's presented to compute highly non-linear terminal embeddings  (embedding strategy (d) above), the green ``InnerProd" curves correspond to the use of nearly linear terminal embeddings (embedding strategy (c) above), and the blue ``Linear" curves correspond to the use of linear JL embedding matrices (embedding strategy (b) above).
}
\label{fig:CompareEmbeddings}
\end{figure}

Looking at Figure~\ref{fig:CompareEmbeddings} one can see that the most non-linear embedding strategy (d) -- i.e., Algorithm~\ref{alg:FTE}
-- allows for the best compressed NN classification performance, outperforming standard linear JL embeddings for all choices of $m$.  Perhaps most interestingly, it also quickly converges to the uncompressed NN classification performance, matching it to within $1$ percent at the values of $m = 24$ for MNIST and $m = 15$ for COIL-100.  This corresponds to relative dimensionality reductions of 
$$100(1 - 24/784) \% \approx 96.9 \% $$
and
$$100(1 - 15/16384) \% \approx 99.9 \%,$$
respectively, with negligible loss of NN classification accuracy.  As a result, it does indeed appear as if nonlinear terminal embeddings have the potential to allow for improvements in dimensionality reduction in the context of classification beyond what standard linear JL embeddings can achieve.

Of course, challenges remain in the practical application of such nonlinear terminal embeddings.  Principally, their computation by, e.g., Algorithm~\ref{alg:FTE} is orders of magnitude slower than simply applying a JL embedding matrix to the data one wishes to compressively classify.  Nonetheless, if dimension reduction at all costs is one's goal, terminal embeddings appear capable of providing better results than their linear brethren.  Recent theoretical work \cite{cherapanamjeri2022terminal} aimed at lessening their computational deficiencies looks promising.

\subsection{Additional Experiments on Effective Distortions and Run Times}

In this section we further investigate the best performing terminal embedding strategy from the previous section (i.e., Algorithm~\ref{alg:FTE}) on the MNIST and COIL-100 data sets.  In particular, we provide illustrative experiments concerning the improvement of $(i)$ compressive classification accuracy with training set size, and $(ii)$ the effective distortion of the terminal embedding with embedding dimension $m+1$.  Furthermore, we also investigate $(iii)$ the run time scaling of Algorithm~\ref{alg:FTE}.

To compute the effective distortions of a given (terminal) embedding of training data $X$, $f: \mathbbm{R}^N \rightarrow \mathbbm{R}^{m+1}$, over all available test and train data $X \cup S$ we use 
\begin{align*}
\text{MaxDist}_f = \underset{{\bf x} \in X}{\max} \; \underset{{\bf y} \in S \cup X \setminus \{ {\bf x} \}}{\max} \dfrac{\|f({\bf y}) - f({\bf x})\|_{2}}{\| {\bf y} - {\bf x}\|_{2}}, \quad \text{MinDist}_f = \underset{{\bf x} \in X}{\min} \; \underset{{\bf y} \in S \cup X \setminus \{ {\bf x} \}}{\min} \dfrac{\|f({\bf y}) - f({\bf x})\|_{2}}{\| {\bf y} - {\bf x}\|_{2}}.
\end{align*}
Note that these correspond to estimates of the upper and lower multiplicative distortions, respectively, of a given terminal embedding (see (\ref{eq: terminal embedding})). In order to better understand the effect of the minimizer ${\bf y}'$ of the minimization problem in Line 3 of Algorithm~\ref{alg:FTE} on the final embedding $f$, we will also separately consider the effective distortions of its component linear JL embedding ${\bf y} \mapsto (\Phi {\bf y}, 0)$ below.  See Figures~\ref{fig:MNISTFigure} and~\ref{fig:COILFigure} for such plots using the MNIST and COIL-100 data sets, respectively.

\begin{figure}[H] 
\centering
\includegraphics[width=1\textwidth]{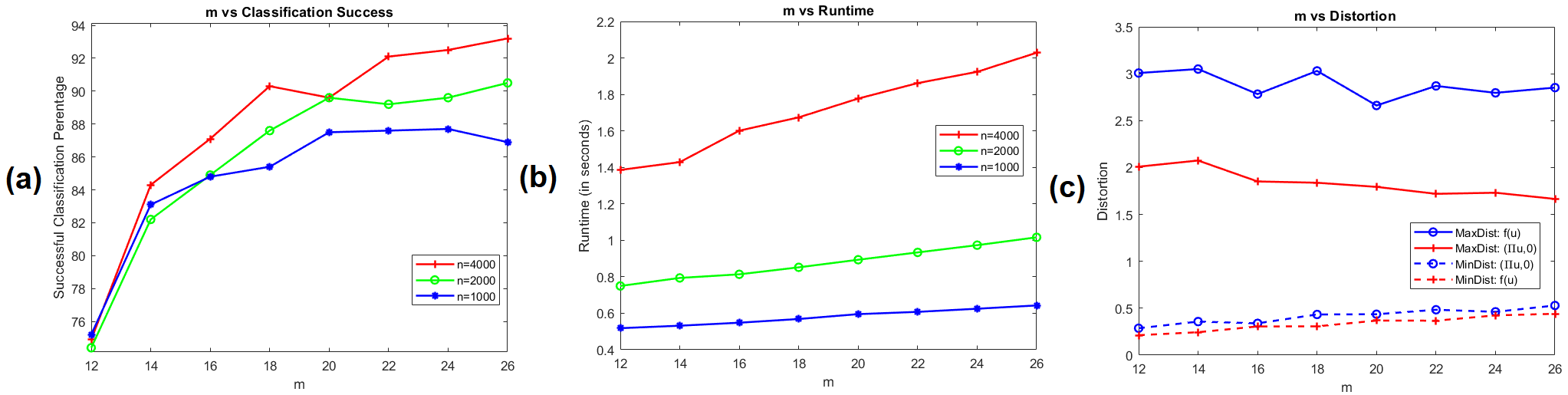}
\caption{This figure compares (a) compressive NN classification accuracies, and (b) the classification run times of Algorithm~\ref{Alg:Class} averaged over all ${\bf y} \in S$, on the MNIST data set.  Three different training data set sizes $n = \lvert X \rvert \in \{ 1000,~ 2000,~ 4000\}$ were fixed as the embedding dimension $m+1$ varied for each of the first two subfigures.  Recall that the test set size is always fixed to $n' = 1000$.  In addition, Figure (c) compares MaxDist$_f$ and MinDist$_f$ for the nonlinear $f$ computed by Algorithm \ref{alg:FTE} versus its component linear embedding ${\bf y} \mapsto (\Phi {\bf y}, 0)$ as $m$ varies for a fixed embedded training set size of $n = 4000$. 
}
\label{fig:MNISTFigure}
\end{figure}

\begin{figure}[H] 
\centering
\includegraphics[width=1\textwidth]{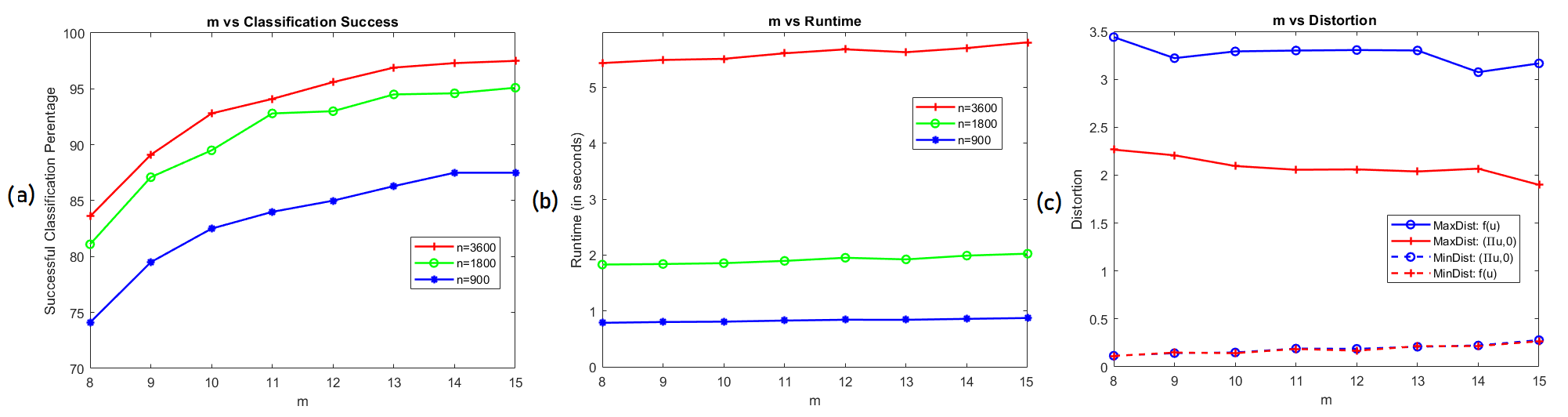}
\caption{Figures (a) and (b) here are run with identical parameters as for their corresponding subfigures in Figure~\ref{fig:MNISTFigure}, except using the COIL-100 data set.  Similarly, Figure (c) compares MaxDist$_f$ and MinDist$_f$ for the nonlinear $f$ computed by Algorithm \ref{alg:FTE} versus its component linear embedding ${\bf y} \mapsto (\Phi {\bf y}, 0)$ as $m$ varies for a fixed embedded training set size of $n = 3600$. 
}
\label{fig:COILFigure}
\end{figure}

Looking at Figures~\ref{fig:MNISTFigure} and~\ref{fig:COILFigure} one notes several consistent trends.  First, compressive classification accuracy increases with both training set size $n$ and embedding dimension $m$, as generally expected.  Second, compressive classification run times also increase with training set size $n$ (as well as more mildly with embedding dimension $m$).  This is mainly due to the increase in the number of constraints in Line 3 of Algorithm~\ref{alg:FTE} with the training set size $n$.  Finally, the distortion plots indicate that the nonlinear terminal embeddings $f$ computed by Algorithm~\ref{alg:FTE} tend to preserve the lower distortions in a similar way as linear JL embeddings do. We obtained different results for upper distortions. Note that our main result Theorem \ref{theo: main} shows that the upper distortions of the nonlinear embeddings considered converge to 1 for sufficiently large dimensions $m$. However, for smaller dimensions the upper distortions of the nonlinear embeddings $f$ appear to significantly exceed that of the linear one. In fact, for small dimensions we obtain an upper distortion around 3, which falls outside the scope of our framework (1 + $\varepsilon$) with $\varepsilon \in (0,1)$. As a result, the nonlinear terminal embeddings considered here appear to spread the initially JL embedded data out, perhaps pushing different classes away from one another in the process.  If so, it would help explain the increased compressive NN classification accuracy observed for Algorithm~\ref{alg:FTE} in Figure~\ref{fig:CompareEmbeddings}.

\section*{Acknowledgements} Mark Iwen was supported in part by NSF DMS 2106472.  Mark Philip Roach was supported in part by NSF DMS 1912706.  Mark Iwen would like to dedicate his effort on this paper to William E. Iwen, 1/27/1939 -- 12/10/2021, as well as to all those at the VA Medical Center in St. Cloud, MN who cared for him so tirelessly during the COVID-19 pandemic.\footnote{\url{https://www.legacy.com/us/obituaries/name/william-iwen-obituary?id=33790586}}  Thank you.

\bibliographystyle{plain}
\bibliography{bibliography}

\end{document}